\documentclass[11pt]{amsart}
\usepackage{amsmath}
\usepackage{cases}
\usepackage{mathrsfs, euscript}
\usepackage{bbm}
\usepackage{amssymb}
\usepackage{txfonts}
\usepackage{amscd}
\usepackage{amsfonts,latexsym,amsmath,amsthm,amsxtra,mathdots,amssymb,latexsym}
\usepackage[all,cmtip]{xy}
\RequirePackage{amsmath} \RequirePackage{amssymb}
\usepackage{color}
\usepackage{colordvi}
\usepackage{multicol}
\usepackage{hyperref}

    \newcommand{\Isom}{{\mathrm{Isom}}}

\newcommand\abs[1]{\left\lvert #1 \right\rvert}

\def\-{^{-1}}

\def\-{^{-1}}

\newcommand{\delete}[1]{}

     \newcommand{\SL}{{\mathrm{SL}}}
     \newcommand{\SO}{{\mathrm{SO}}}

    \newcommand{\st}{{\mathrm{st}}}

    \theoremstyle{plain}

\newtheorem*{main}{Main Theorem}

    \newtheorem{thm}{Theorem}[section] \newtheorem{cor}[thm]{Corollary}
    \newtheorem{lem}[thm]{Lemma}  \newtheorem{prop}[thm]{Proposition}
    
     \newtheorem{defn}[thm]{Definition}

    \newtheorem {claim}[thm]{Claim}
    \numberwithin{equation}{section}

\theoremstyle{remark}
\newtheorem*{remark}{Remark}

\begin{document}

\title{Barycentric straightening and bounded cohomology}

\author[Jean-Fran\c{c}ois Lafont]{Jean-Fran\c{c}ois Lafont$^\dagger$}
\address{Department of Mathematics,
The Ohio State University, 
231 W. 18th Ave.,
Columbus, OH 43210, U.S.A.}
\email{jlafont@math.ohio-state.edu}

\author[Shi Wang]{Shi Wang}
\address{Department of Mathematics,
Indiana University,
831 E. Third St.,
Bloomington, IN 47405, U.S.A.}
\email{wang679@iu.edu}

\thanks{$^\dagger$ The work of the first author is partially supported by the NSF, under grants DMS-1510640, DMS-1812028.}

\keywords{Barycenter method, bounded cohomology, semisimple Lie group, Dupont's problem.}
\subjclass{57T10 (primary), 53C35 (secondary)}

\begin{abstract}
We study the barycentric straightening of simplices in higher rank irreducible symmetric
spaces of non-compact type. We show that, for an $n$-dimensional symmetric space of rank $r\geq 2$
(excluding $SL(3,\mathbb R)/SO(3)$ and $SL(4, \mathbb R)/SO(4)$), the $p$-Jacobian has uniformly
bounded norm, provided $p\geq n-r+2$. As a consequence, for the corresponding
non-compact, connected, semisimple real Lie group $G$, in degrees $p\geq n-r+2$, every degree $p$ cohomology class
has a bounded representative. This answers Dupont's problem in small codimension. We also give examples of
symmetric spaces where the barycentrically straightened simplices of dimension $n-r$ have unbounded volume, showing
that the range in which we obtain boundedness of the $p$-Jacobian is very close to optimal.
\end{abstract}

\maketitle

%%%%%%%%%%%%%%%%%%%%%%%%%%%%%%%%%%%%%%%%%%
%%%%%%%%%%%%%%%%%%%%%%%%%%%%%%%%%%%%%%%%%%

\section{Introduction}

When studying the bounded cohomology of groups, an important theme is the comparison map from bounded cohomology
to ordinary cohomology. In the context of non-compact, connected, semisimple Lie groups, Dupont raised the question
of whether this comparison map is always surjective \cite{Du2} (see also Monod's ICM address \cite[Problem A']{Mo},
and \cite[Conjecture 18.1]{BIMW}). Properties of these Lie groups $G$ are closely related to properties of the
corresponding non-positively curved symmetric space $X=G/K$. Geometric methods on the space $X$ can often be
used to recover information about the Lie group $G$. This philosophy was used by Lafont and Schmidt \cite{LS} to
show that the comparison map is surjective in degree $\dim(X)$. In the present paper, we extend this result to smaller
degrees, and show:

\begin{main}
Let $X=G/K$ be an $n$-dimensional irreducible symmetric space of non-compact type of rank $r=rank(X)\geq 2$, excluding $SL(3,\mathbb R)/SO(3)$ and $SL(4, \mathbb R)/SO(4)$, and
$\Gamma$ a cocompact torsion-free lattice in $G$. Then the comparison maps
$\eta:H^{*}_{c,b}(G,\mathbb{R})\rightarrow H^{*}_c(G,\mathbb{R})$ and $\eta':H^{*}_{b}(\Gamma,\mathbb{R})\rightarrow
H^{*}(\Gamma,\mathbb{R})$ are both surjective in all degrees $* \geq n-r+2$.
\end{main}

The idea of the proof is similar to that in \cite{LS}. One defines a barycentric straightening of simplices in $X$, and
uses it to construct bounded cocycles representing any given cohomology class. These cocycles are obtained by integrating
a suitable differential form on various straightened simplices. Since the differential form has bounded norm, the key step
is to show that the Jacobian of the straightened simplex is uniformly controlled (independent of the simplex or the point
in it). Showing this later property requires some work, and is done in Sections \ref{sec:Jacobian} and
\ref{sec:combinatorial-problem} (following the general approach of Connell and Farb \cite{CF1}, \cite{CF2}).
The proof of the {\bf  Main Theorem} is then given in Section \ref{sec:surjectivity}.

%\begin{remark}
%The result in Lafont-Schmidt relied heavily on work of Connell and Farb \cite{CF1}. {\color{red}
%In order to establish control on the
%norm of the $p$-Jacobian, the key technical result requires extending the Connell-Farb ``eigenvalue matching'' argument
%to lower degrees -- see Theorem \ref{thm:bound-ratio}.}
%\end{remark}

\begin{remark}
For the various families of higher rank symmetric spaces, the dimension grows roughly quadratically in the rank. Our {\bf Main
Theorem} thus answers Dupont's question for continuous cohomology classes in degree close to the dimension of
the symmetric space.
Prior results on this problem include some work on the degree two case (Domic and Toledo \cite{DT}, as well as
Clerk and Orsted \cite{CO}) as well as the top-degree case (Lafont and Schmidt \cite{LS}).
In his seminal paper on the subject, Gromov
showed that characteristic classes of flat bundles are bounded classes \cite{Gr}. Using Gromov's result,
Hartnick and Ott \cite{HO} were able to obtain complete answers for several specific classes of Lie groups (e.g. of Hermitian type, as well as some other cases).

The recently posted preprint \cite{KK} of Inkang Kim and Sungwoon Kim uses similar methods to obtain
uniform control of the Jacobian in codimension one.
Their paper also contains a wealth of other applications, which we have not pursued in the present paper. On the
other hand, their results do not produce any new bounded cohomology classes (since in the higher rank case, the
codimension one continuous cohomology always vanishes).
\end{remark}

\vskip 5pt

\centerline{\bf Acknowledgments}

\vskip 5pt

We would like to thank Michelle Bucher-Karlsson, Marc Burger, Chris Connell, Tobias Hartnick, Clara L\"oh, Ben McReynolds, and Roman Sauer for their helpful comments.

%%%%%%%%%%%%%%%%%%%%%%%%%%%%%%%%%%%%%%%%%%
%%%%%%%%%%%%%%%%%%%%%%%%%%%%%%%%%%%%%%%%%%

\section{Preliminaries}\label{sec:preliminaries}
\subsection{Symmetric spaces of non-compact type}
In this section, we give a quick review of some results on symmetric spaces of non-compact type; for more details,
we refer the reader to  Eberlein's book \cite{Eb}. Let $X=G/K$ be a symmetric space of non-compact type,
where $G$ is semisimple and $K$ is a maximal compact subgroup of $G$. Geometrically $G$ can be identified with
$\text{Isom}_0(X)$, the connected component of the isometry group of $X$ that contains the identity, and
$K=\text{Stab}_p(G)$ for
some $p\in X$. Fixing a basepoint $p\in X$, we have a Cartan decomposition $\mathfrak{g}=\mathfrak{k}+\mathfrak{p}$
of the Lie algebra $\mathfrak{g}$ of $G$, where $\mathfrak{k}$ is the Lie algebra of $K$, and $\mathfrak{p}$ can be
isometrically identified with $T_pX$ using the Killing form. Let $\mathfrak{a}\subseteq \mathfrak{p}$ be a maximal
abelian subalgebra of $\mathfrak{p}$. We can identify $\mathfrak{a}$ with the tangent space of a flat $\mathcal{F}$
at $p$ -- that is to say, an isometrically embedded Euclidean space $\mathbb{R}^r\subseteq X$, where $r$ is the
rank of $X$. Given any vector $v\in T_pX$, there exists a flat $\mathcal{F}$ that is tangent to $v$. We say $v$ is
regular if such a flat is unique, and singular otherwise.

Now let $v\in \mathfrak{p}$ be a regular vector. This direction defines a point $v(\infty)$ on the visual boundary
$\partial X$ of $X$. $G$ acts on the visual boundary $\partial X$. The orbit set $Gv(\infty)=\partial_FX\subseteq \partial X$
is called a Furstenberg boundary of $X$. Since both $G$ and $K$ act transitively on $\partial_FX$, $\partial_FX$ is
compact. In fact, a point stabilizer for the $G$-action on $\partial _FX$ is a minimal parabolic subgroup $P$, so we can
also identify $\partial_FX$ with the quotient $G/P$. In the rest of this paper, we will use a specific realization
of the Furstenberg boundary -- the one given by choosing the regular vector $v$ to point towards a barycenter of a
Weyl chamber in the flat.

For each element $\alpha$ in the dual space $\mathfrak{a}^*$ of $\mathfrak{a}$, we define
$\mathfrak{g}_\alpha=\{Y\in \mathfrak{g}\;|\;[A,Y]=\alpha(A)Y \;\textrm{for all}\; A\in \mathfrak{a}\}$.
We call $\alpha$ a root if $\mathfrak{g}_\alpha$ is nontrivial, and in such case we call $\mathfrak{g}_\alpha$
the root space of $\alpha$. We denote the finite set of roots $\Lambda$, and we have the following root space
decomposition
$$\mathfrak{g}=\mathfrak{g}_0\oplus \bigoplus_{\alpha\in\Lambda}\mathfrak{g}_\alpha$$
where $\mathfrak{g}_0=\{Y\in \mathfrak{g}\;|\;[A,Y]=0\;\textrm{for all}\; A\in \mathfrak{a}\}$, and the
direct sum is orthogonal with respect to the canonical inner product on $\mathfrak{g}$.

Let $\theta$ be the Cartan involution at the point $p$. Then $\theta$ is an involution
on $\mathfrak{g}$, which acts by $I$ on $\mathfrak{k}$ and $-I$ on $\mathfrak{p}$,
hence it preserves Lie bracket. We can define $\mathfrak{k}_\alpha=(I+\theta)\mathfrak{g}_\alpha\subseteq \mathfrak{k}$,
and $\mathfrak{p}_\alpha=(I-\theta)\mathfrak{g}_\alpha\subseteq \mathfrak{p}$, with the following properties:

\begin{prop}\cite[Proposition 2.14.2]{Eb}\label{prop:Eberlein}
$(1)\;I+\theta:\mathfrak{g}_\alpha\rightarrow \mathfrak{k}_\alpha$ and $I-\theta:\mathfrak{g}_\alpha\rightarrow \mathfrak{p}_\alpha$ are linear isomorphisms. Hence $\dim(\mathfrak{k}_\alpha)=\dim(\mathfrak{g}_\alpha)=\dim(\mathfrak{p}_\alpha)$.\\
$(2)\;\mathfrak{k}_\alpha=\mathfrak{k}_{-\alpha}$ and $\mathfrak{p}_\alpha=\mathfrak{p}_{-\alpha}$ for all $\alpha\in \Lambda$, and $\mathfrak{k}_\alpha\oplus\mathfrak{p}_\alpha=\mathfrak{g}_\alpha\oplus \mathfrak{g}_{-\alpha}$.\\
$(3)\;\mathfrak{k}=\mathfrak{k}_0\oplus \bigoplus_{\alpha\in \Lambda^+}\mathfrak{k}_\alpha$ and $\mathfrak{p}=\mathfrak{a}\oplus \bigoplus_{\alpha\in \Lambda^+}\mathfrak{p}_\alpha$, where $\mathfrak{k}_0=\mathfrak{g}_0\cap \mathfrak{k}$,
and  $\Lambda^+$ is the set of positive roots.
\end{prop}

\begin{remark}
Since $\mathfrak{p}_\alpha=(\mathfrak{g}_\alpha+\mathfrak{g}_{-\alpha})\cap \mathfrak{p}$, the direct sum of
$\mathfrak{p}$ in $(3)$ of Proposition \ref{prop:Eberlein} is also orthogonal with respect to the canonical inner
product on $\mathfrak{p}$.
\end{remark}

We now analyze the adjoint action of $\mathfrak{k}$ on $\mathfrak{a}$. Let $u\in \mathfrak{k}_\alpha$ and $v\in \mathfrak{a}$, we can write $u$ as $(I+\theta)w$ where $w\in \mathfrak{g}_\alpha$, hence we have
\begin{align*}
 [u,v]&=[(I+\theta)w,v]=[w,v]+[\theta w,v] =-\alpha(v)w+\theta[w,-v] \\
 &=-\alpha(v)w+\theta (\alpha(v)w)=-\alpha(v)(I-\theta)(w)\\
 &=-\alpha(v)(I-\theta)(I+\theta)^{-1}u
\end{align*}
This gives the following proposition.

\begin{prop}\label{prop:Lie algebra action}
Let $\alpha$ be a root. The adjoint action of $\mathfrak{k}_\alpha$ on $\mathfrak{a}$ is given by
$$[u,v]=-\alpha(v)(I-\theta)(I+\theta)^{-1}u$$
for any $u\in \mathfrak{k}_\alpha$ and $v\in \mathfrak{a}$. In particular, $\mathfrak{k}_\alpha$ maps $v$ into $\mathfrak{p}_\alpha$.
\end{prop}

Assume $v\in \mathfrak{a}\subseteq T_xX$ is inside a fixed flat through $x$, and let $K_v$ be the stabilizer of $v$ in $K$.
Then the space $K_v\mathfrak{a}$ is the tangent space of the union of all flats that goes through $v$. Equivalently, it is
the union of all vectors that are parallel to $v$, hence it can be identified with
$\mathfrak{a}\oplus\bigoplus_{\alpha\in \Lambda^+,\alpha(v)=0}\mathfrak{p}_\alpha$. In particular, if $v$ is regular,
then the space is just $\mathfrak{a}$. Moreover, if we denote by $\mathfrak{k}_v$ the Lie algebra of $K_v$, then
$\mathfrak{k}_v=\{u\in \mathfrak{k}\;|\;[u,v]=0\}=\mathfrak{k}_0 \oplus\bigoplus_{\alpha\in \Lambda^+,\alpha(v)=0}\mathfrak{k}_\alpha$.

\subsection{Patterson-Sullivan measures}
Let $X=G/K$ be a symmetric space of non-compact type, and $\Gamma$ be a cocompact lattice in G. In \cite{Al}, Albuquerque generalizes the construction of Patterson-Sullivan to higher rank symmetric spaces. He showed that for each $x\in X$, we can assign a probability measure $\mu(x)$ that is $G$-equivariant and is fully supported on the Furstenberg boundary $\partial_F(X)$. Moreover, for $x,y\in X$ and $\theta\in \partial_F(X)$, the Radon-Nikodym derivative is given by
$$\frac{d\mu(x)}{d\mu(y)}(\theta)=e^{hB(x,y,\theta)}$$
where $h$ is the volume entropy of $X/\Gamma$, and $B(x,y,\theta)$ is the Busemann function on $X$. Recall that, in a non-positively curved space $X$, the Busemann function $B$ is defined by
$$B(x,y,\theta)=\lim_{t\rightarrow\infty}(d_X(y,\gamma_\theta(t))-t)$$
where $\gamma_\theta$ is the unique geodesic ray from $x$ to $\theta$. Fixing a basepoint $O$ in $X$, we shorten $B(O, y,\theta)$ to just $B(y,\theta)$. Notice that for fixed $\theta\in \partial_F(X)$ the Busemann function is convex on $X$, and by integrating on $\partial_F(X)$, we obtain, for any probability measure $\nu$ that is fully supported on the Furstenberg boundary $\partial_FX$, a strictly convex function
$$x\mapsto \int_{\partial_F X}B(x,\theta)d\nu(\theta)$$
(See \cite[Proposition \ref{thm:bound-ratio}]{CF1} for a proof of this last statement.)

Hence we can define the barycenter $bar(\nu)$ of $\nu$ to be the unique point in $X$ where the function attains its minimum. It is clear that this definition is independent of the choice of basepoint $O$.

\subsection{Barycenter method}\label{sec:barycenter method}
In this section, we discuss the barycentric straightening introduced by Lafont and Schmidt \cite{LS} (based on the barycenter method originally developed by Besson, Courtois, and Gallot \cite{BCG}). Let $X=G/K$ be a symmetric space of non-compact type, and $\Gamma$ be a cocompact lattice in G. We denote by $\Delta^k_s$ the standard spherical k-simplex in the Euclidean space, that is
$$\Delta^k_s=\Big\{(a_1,\ldots ,a_{k+1})\mid a_i\geq 0, \sum_{i=1}^{k+1}a_i^2=1\Big\}\subseteq \mathbb{R}^{k+1},$$
with the induced Riemannian metric from $\mathbb{R}^{k+1}$, and with ordered vertices $(e_1,\ldots,e_{k+1})$. Given any singular k-simplex $f:\Delta^k_s\rightarrow X$, with ordered vertices $V=(x_1,\ldots,x_{k+1})=\left(f(e_1),\ldots,f(e_{k+1})\right)$, we define the k-straightened simplex
$$\st_k(f):\Delta^k_s\rightarrow X$$
$$st_k(f)(a_1,\ldots ,a_{k+1}):=bar\left(\sum_{i=1}^{k+1}a_i^2\mu(x_i)\right)$$
where $\mu(x_i)$ is the Patterson-Sullivan measure at $x_i$. We notice that $st_k(f)$ is determined by the (ordered) vertex set $V$, and we denote $st_k(f)(\delta)$ by $st_V(\delta)$, for $\delta\in\Delta^k_s$.

Observe that the map $\st_k(f)$ is $C^1$, since one can view this map as the restriction of the $C^1$-map $\st_n(f)$ to a k-dimensional subspace (see e.g. \cite[Property (3)]{LS}). For any $\delta=\sum_{i=1}^{k+1}a_ie_i\in \Delta_s^{k}$, $\st_k(f)(\delta)$ is defined to be the unique point where the function
$$x\mapsto \int_{\partial_F X}B(x,\theta)d\left(\sum_{i=1}^{k+1}a_i^2\mu(x_i)\right)(\theta)$$
is minimized. Hence, by differentiating at that point, we get the 1-form equation
$$\int_{\partial_F X}dB_{(st_V(\delta),\theta)}(\cdot)d\left(\sum_{i=1}^{k+1}a_i^2\mu(x_i)\right)(\theta)\equiv 0$$
which holds identically on the tangent space $T_{st_V(\delta)}X$.
Differentiating in a direction $u\in T_\delta(\Delta_s^k)$ in the source, one obtains the $2$-form equation
\begin{equation}\label{eqn:2-form}
\sum_{i=1}^{k+1}2a_i\langle u,e_i\rangle_\delta\int_{\partial_F X}dB_{(st_V(\delta),\theta)}(v)d(\mu(x_i))(\theta)
+\int_{\partial_F X}DdB_{(st_V(\delta),\theta)}(D_\delta(st_V)(u),v)d\left(\sum_{i=1}^{k+1}a_i^2\mu(x_i)\right)(\theta)\equiv 0
\end{equation}
which holds for every $u\in T_\delta(\Delta_s^k)$ and $v\in T_{st_V(\delta)}(X)$.
Now we define two semi-positive definite quadratic forms $Q_1$ and $Q_2$ on $T_{st_V(\delta)}(X)$:
$$Q_1(v,v)=\int_{\partial_F X}dB^2_{(st_V(\delta),\theta)}(v)d\left(\sum_{i=1}^{k+1}a_i^2\mu(x_i)\right)(\theta)$$
$$Q_2(v,v)=\int_{\partial_F X}DdB_{(st_V(\delta),\theta)}(v,v)d\left(\sum_{i=1}^{k+1}a_i^2\mu(x_i)\right)(\theta)$$
In fact, $Q_2$ is positive definite since $\sum_{i=1}^{k+1}a_i^2\mu(x_i)$ is fully supported on $\partial_F X$ (see \cite[Section 4]{CF1}). From Equation (\ref{eqn:2-form}), we obtain, for $u\in T_\delta (\Delta _s^k)$ a unit vector
and $v\in T_{st_V(\delta)}(X)$ arbitrary, the following
\begin{align}
\abs{Q_2(D_\delta(st_V)(u),v)}&=\abs{-\sum_{i=1}^{k+1}2a_i\langle u,e_i\rangle_\delta\int_{\partial_F X}dB_{(st_V(\delta),\theta)}(v)d(\mu(x_i))(\theta)} \label{eqn:Q1-bounds-Q2}\\
& \leq \left(\sum_{i=1}^{k+1}\langle u,e_i\rangle_\delta^2\right)^{1/2}\left(\sum_{i=1}^{k+1}4a_i^2\left(\int_{\partial_F X}dB_{(st_V(\delta),\theta)}(v)d(\mu(x_i))(\theta)\right)^2\right)^{1/2} \nonumber \\
&\leq 2\left(\sum_{i=1}^{k+1}a_i^2\int_{\partial_F X}dB^2_{(st_V(\delta),\theta)}(v)d(\mu(x_i))(\theta)\int_{\partial_F X}1d(\mu(x_i))\right)^{1/2} \nonumber \\
&=2Q_1(v,v)^{1/2} \nonumber
\end{align}
via two applications of the Cauchy-Schwartz inequality.

We restrict these two quadratic forms to the subspace $S=Im(D_\delta(st_V))\subseteq T_{st_V(\delta)}(X)$, and denote the corresponding $k$-dimensional endomorphisms by $H_\delta$ and $K_\delta$, that is
$$Q_1(v,v)=\langle H_\delta(v),v\rangle_{st_V(\delta)}$$
$$Q_2(v,v)=\langle K_\delta(v),v\rangle_{st_V(\delta)}$$
for all $v\in S$.

For points $\delta\in\Delta_s^k$ where $st_V$ is nondegenerate, we now pick orthonormal bases $\{u_1,\ldots ,u_k\}$ on $T_\delta(\Delta_s^k)$, and $\{v_1,\ldots ,v_k\}$ on $S\subseteq T_{st_V(\delta)}(X)$. We choose these so that $\{v_i\}_{i=1}^k$ are eigenvectors of $H_\delta$, and $\{u_1,\ldots, u_k\}$ is the resulting basis obtained by applying the orthonormalization process to the collection of pullback vectors $\{(K_\delta\circ D_\delta(st_V))^{-1}(v_i)\}_{i=1}^k$. So we obtain
\begin{align*}
\det(Q_2|_S)\cdot \abs{Jac_\delta(st_V)} &=\abs{\det(K_\delta)\cdot Jac_\delta(st_V)} \\
& =\abs{\det(\langle K_\delta\circ D_\delta(st_V)(u_i),v_j\rangle)}
\end{align*}
By the choice of bases, the matrix $(\langle K_\delta\circ D_\delta(st_V)(u_i),v_j\rangle)$ is upper triangular, so we have
\begin{align*}
\abs{\det(\langle K_\delta\circ D_\delta(st_V)(u_i),v_j\rangle)}& =\abs{\prod_{i=1}^k\langle K_\delta\circ D_\delta(st_V)(u_i),v_i\rangle} \\
& \leq \prod_{i=1}^k 2\langle H_\delta(v_i),v_i\rangle^{1/2} \\
&=2^k\det(H_\delta)^{1/2}=2^k\det(Q_1|_S)^{1/2}
\end{align*}
where the middle inequality is obtained via Equation (\ref{eqn:Q1-bounds-Q2}).
Hence we get the inequality
$$\abs{Jac_\delta(st_V)}\leq 2^k \cdot \frac{\det(Q_1|_S)^{1/2}}{\det(Q_2|_S)}$$
We summarize the above discussion into the following proposition.
\begin{prop}\label{prop:reduction}
Let $Q_1$, $Q_2$ be the two positive semidefinite quadratic forms defined as above (note $Q_2$ is actually positive definite).
Assume there exists a constant $C$ that only depends on $X$, with the property that
$$\frac{\det(Q_1|_S)^{1/2}}{\det(Q_2|_S)}\leq C$$
for any k-dimensional subspace $S\subseteq T_{st_V(\delta)}X$. Then the quantity $|Jac(st_V)(\delta)|$
is universally bounded -- independent of the choice of $(k+1)$-tuple of points $V\subset X$, and of the point $\delta \in \Delta_s^k$.
\end{prop}

%%%%%%%%%%%%%%%%%%%%%%%%%%%%%%%%%%%%%%%%%%%
%%%%%%%%%%%%%%%%%%%%%%%%%%%%%%%%%%%%%%%%%%%
%%%%%%%%%%%%%%%%%%%%%%%%%%%%%%%%%%%%%%%%%%%

\section{Jacobian Estimate}\label{sec:Jacobian}
Let $X=G/K$ be an irreducible symmetric space of non-compact type. We fix an arbitrary point $x\in X$ and identify $T_xX$ with $\mathfrak{p}$. Let $\mu$ be a probability measure that is fully supported on the Furstenberg boundary $\partial_FX$. Using the same notation as in Section 2.3, we define a semi-positive definite quadratic form $Q_1$ and a positive definite quadratic form $Q_2$ on $T_xX$
$$Q_1(v,v)=\int_{\partial_F X}dB^2_{(x,\theta)}(v)d\mu(\theta)$$
$$Q_2(v,v)=\int_{\partial_F X}DdB_{(x,\theta)}(v,v)d\mu(\theta)$$
for $v\in T_x(X)$.
We will follow the techniques of Connell and Farb \cite{CF1}, \cite{CF2}, and show the following theorem.

\begin{thm}\label{thm:bound-ratio}
Let $X$ be an irreducible symmetric space of non-compact type excluding $\SL(3,\mathbb R)/\SO(3)$ and $\SL(4,\mathbb{R})/\SO(4)$, and let $r=rank(X)\geq 2$. If $n=\dim(X)$, then there exists a constant $C$ that only depends on $X$, such that
$$\frac{\det(Q_1|_S)^{1/2}}{\det(Q_2|_S)}\leq C$$
for any subspace $S\subseteq T_xX$ with $n-r+2\leq \dim(S)\leq n$.
\end{thm}

 In view of Proposition \ref{prop:reduction}, this implies that the barycentrically straightened
simplices of dimension $\geq n-r+2$ have uniformly controlled Jacobians. The reader whose primary interest is
bounded cohomology, and who is willing to take Theorem \ref{thm:bound-ratio} on faith, can skip ahead to Section
\ref{sec:surjectivity} for the proof of the {\bf Main Theorem}.

The rest of this Section will be devoted to the proof of Theorem \ref{thm:bound-ratio}. In Section
\ref{subsec:simplify-quad-form}, we explain some simplifications of the quadratic forms, allowing us to give geometric
interpretations for the quantities involved in Theorem \ref{thm:bound-ratio}. In Section \ref{subsec:eigval-matching},
we formulate the ``weak eigenvalue matching'' Theorem \ref{thm:weak eigenvalue matching} (which will be
established in Section \ref{sec:combinatorial-problem}). Finally, in Section
\ref{subsec:proof-theorem-bounded-ratio}, we will deduce Theorem \ref{thm:bound-ratio}
from Theorem \ref{thm:weak eigenvalue matching}.

\subsection{Simplifying the quadratic forms}\label{subsec:simplify-quad-form}
Following \cite[Section 4.3]{CF1}, we fix a flat $\mathcal{F}$ going through $x$, and denote the tangent space by $\mathfrak{a}$, so $\dim(\mathfrak{a})=r$ is the rank of $X$. By abuse of notation, we identify $\mathfrak{a}$ with $\mathcal{F}$. Choose an orthonormal basis $\{e_i\}$ on $T_xX$ such that $\{e_1,...,e_r\}$ spans $\mathcal{F}$, and assume $e_1$ is regular so that $e_1(\infty)\in \partial_F X$. Then $Q_1$, $Q_2$ can be expressed in the following matrix forms.
$$Q_1=\int_{\partial_F X}O_\theta\begin{pmatrix}
                                   1 & 0\\
                                   0 & 0^{(n-1)} \\
                                 \end{pmatrix}O_\theta^\ast
d\mu(\theta)$$
$$Q_2=\int_{\partial_F X}O_\theta\begin{pmatrix}
                                   0^{(r)} & 0\\
                                   0 & D_\lambda^{(n-r)} \\
                                 \end{pmatrix}O_\theta^\ast
d\mu(\theta)$$
where $D_\lambda=diag(\lambda_1,...,\lambda_{(n-r)})$, and $O_\theta$ is the orthogonal matrix corresponding to the unique element in $K$ that sends $e_1$ to $v_{(x,\theta)}$, the direction at $x$ pointing towards $\theta$. Moreover, there exists a constant $c>0$ that only depends on $X$, so that $\lambda_i\geq c$ for $1\leq i\leq n-r$. For more details, we refer the readers to the original \cite{CF1}.

Denote by $\bar Q_2$ the quadratic form given by
$$\bar Q_2=\int_{\partial_F X}O_\theta\begin{pmatrix}
                                   0^{(r)} & 0\\
                                   0 & I^{(n-r)} \\
                                 \end{pmatrix}O_\theta^\ast
d\mu(\theta)$$
Then the difference $Q_2 - c\bar Q_2$ is positive semi-definite, hence $\det(Q_2|_S) \geq \det(c \bar Q_2|_S)$.
So in order to show Theorem \ref{thm:bound-ratio}, it suffices to assume $Q_2$ has the matrix form
$$\int_{\partial_F X}O_\theta\begin{pmatrix}
                                   0^{(r)} & 0\\
                                   0 & I^{(n-r)} \\
                                 \end{pmatrix}O_\theta^\ast
d\mu(\theta)$$
Given any $v\in T_xX$, we have the following geometric estimates on the value of the quadratic form
\begin{align}
Q_1(v,v)&=\int_{\partial_F X}v^tO_\theta\begin{pmatrix}
                                   1 & 0\\
                                   0 & 0^{(n-1)} \\
                                 \end{pmatrix}O_\theta^\ast v \hskip 2pt d\mu(\theta) \label{eqn:Q1-formula}\\
&=\int_{\partial_F X}\langle O_\theta^\ast v,e_1\rangle^2d\mu(\theta) \nonumber \\
&\leq \int_{\partial_F X}\sum_{i=1}^r\langle O_\theta^\ast v,e_i\rangle^2d\mu(\theta) \nonumber \\
&=\int_{\partial_F X}\sin^2(\angle (O_\theta^\ast v,\mathcal{F}^\perp))d\mu(\theta) \nonumber
\end{align}
Roughly speaking, $Q_1(v,v)$ is bounded above by the weighted average of the time the $K$-orbit spends away from
$\mathcal{F}^\perp$. Similarly we can estimate
\begin{align}
Q_2(v,v)&=\int_{\partial_F X}v^tO_\theta\begin{pmatrix}
                                   0^{(r)} & 0\\
                                   0 & I^{(n-r)} \\
                                 \end{pmatrix}O_\theta^\ast v \hskip 2pt
d\mu(\theta) \label{eqn:Q2-formula}\\
&=\int_{\partial_F X}\sum_{i=r+1}^n\langle O_\theta^\ast v,e_i\rangle^2d\mu(\theta) \nonumber \\
&=\int_{\partial_F X}\sin^2(\angle (O_\theta^\ast v,\mathcal{F}))d\mu(\theta) \nonumber
\end{align}
So again, $Q_2(v,v)$ roughly measures the weighted average of the time the $K$-orbit spends away from $\mathcal {F}$.

\subsection{Eigenvalue matching}\label{subsec:eigval-matching}
In their original paper, Connell and Farb showed an eigenvalue matching theorem \cite[Theorem 4.4]{CF1}, in order to get the Jacobian estimate in top dimension. For the small eigenvalues of $Q_2$ (there are at most $r$ of them), they want to find twice as many comparatively small eigenvalues of $Q_1$. Then by taking the product of those eigenvalues, they obtain a uniform upper bound on the ratio of determinants $\det (Q_1)^{1/2}/\det (Q_2)$, which yields an upper bound on the Jacobian.
However, as was pointed out by Inkang Kim and Sungwoon Kim, there was a mistake in the proof. Connell and Farb fixed the gap by showing a weak eigenvalue matching theorem \cite[Theorem 0.1]{CF2}, which was sufficient to imply the Jacobian inequality.

We generalize this method and show that in fact we can find $(r-2)$ additional small eigenvalues of
$Q_1$ that are bounded by a universal constant times the smallest eigenvalue of $Q_2$. This allows for the
Jacobian inequality to be maintained when we pass down to a subspace of codimension at most $(r-2)$.
We now state our version of the weak eigenvalue matching theorem.

\begin{defn}
We call a set of unit vectors $\{w_1,...,w_k\}$ a $\delta$-orthonormal $k$-frame if $\langle w_i,w_j\rangle< \delta$ for all $1\leq i<j\leq k$.
\end{defn}

\begin{thm}(Weak eigenvalue matching.)\label{thm:weak eigenvalue matching}
 Let $X$ be an irreducible symmetric space of non-compact type, with $r=rank(X)\geq 2$, excluding $\SL(3,\mathbb R)/\SO(3)$ and $\SL(4,\mathbb{R})/\SO(4)$. There exist constants $C'$, $C$, $\delta$ that only depend on $X$ so that the following holds. Given any $\epsilon<\delta$, and any orthonormal $k$-frame $\{v_1,...,v_k\}$ in $T_xX$ with $k\leq r$, whose span $V$ satisfies $\angle(V,\mathcal{F})\leq \epsilon$, then there is a $(C'\epsilon)$-orthonormal $(2k+r-2)$-frame given by vectors $\{v_1',v_1'',...,v_1^{(r)},v_2',v_2'',...,v_k',v_k''\}$, such that for $i=1,...,k$, and $j=1,...,r$, we have the following inequalities:
$$\angle(hv_i',\mathcal{F}^\perp)\leq C\angle(hv_i,\mathcal{F})$$
$$\angle(hv_i'',\mathcal{F}^\perp)\leq C\angle(hv_i,\mathcal{F})$$
$$\angle(hv_1^{(j)},\mathcal{F}^\perp)\leq C\angle(hv_1,\mathcal{F})$$
for all $h\in K$, where $hv$ is the linear action of $h\in K$ on $v\in T_xX\simeq \mathfrak{p}$.
\end{thm}

 The proof of Theorem \ref{thm:weak eigenvalue matching} will be delayed to Section \ref{sec:combinatorial-problem}.

\subsection{Proof of Theorem \ref{thm:bound-ratio}}\label{subsec:proof-theorem-bounded-ratio}
In this section, we will prove Theorem \ref{thm:bound-ratio} using Theorem \ref{thm:weak eigenvalue matching}. Before starting the proof, we will need the following three elementary results from linear algebra.

\begin{lem}\label{lem:compare-eigvals}
Let $Q$ be a positive definite quadratic form on some Euclidean space $V$ of dimension $n$, with eigenvalues $\lambda_1\leq \lambda_2\leq...\leq \lambda_n$. Let $W\subseteq V$ be a subspace of codimension $l$, and let $\mu_1\leq \mu_2\leq...\leq\mu_{n-l}$ be the eigenvalues of $Q$ restricted to $W$. Then $\lambda_i\leq\mu_i\leq \lambda_{i+l}$ holds for $i=1,\ldots ,n-l$.
\end{lem}
\begin{proof} We argue by contradiction. Assume $\mu_i> \lambda_{i+l}$ for some $i$. Take the subspace $W_0\subseteq W$ spanned by the eigenvectors corresponding to $\mu_i, \mu_{i+1},\ldots ,\mu_{n-l}$; clearly $\dim(W_0) = n-l-i+1$. So for any nonzero vectors $v\in W_0$, we have $Q(v,v)\geq \mu_i\| v\|^2>\lambda_{i+l}\| v\|^2$. However, if we denote $V_0\subseteq V$ the $(i+l)$-dimensional subspace spanned by the eigenvectors corresponding to $\lambda_1,\ldots,\lambda_{i+l}$, we have $Q(v,v)\leq \lambda_{i+l}\| v\|^2$ for any $v\in V_0$. But $\dim(W_0\cap V_0)\geq \dim(W_0)+\dim(V_0)-\dim(V)=1$ implies $W_0\cap V_0$ is nontrivial, so we obtain a contradiction. This establishes $\mu_i\leq \lambda_{i+l}$. A similar argument shows $\lambda_i\leq\mu_i$.
\end{proof}

\begin{lem}\label{lem:using-other-bases}
Let $Q$ be a positive definite quadratic form on some Euclidean space $V$ of dimension $n$, with eigenvalues $\lambda_1\leq \lambda_2\leq \cdots \leq \lambda_n$. If $\{v_1,\ldots ,v_n\}$ is any orthonormal frame of $V$, ordered so that $Q(v_1,v_1)\leq Q(v_2,v_2)\leq \cdots \leq Q(v_n,v_n)$, then $Q(v_i,v_i)\geq \lambda_i/n$ for $i=1,\ldots ,n$.
\end{lem}

\begin{proof} We show this by induction on the dimension of $V$. The statement is clear when $n=1$, so let us now assume we have the statement for $\dim(V)=n-1$. Now if $\dim(V)=n$, we restrict the quadratic form $Q$ to the $(n-1)$-dimensional subspace $W$ spanned by $v_1,\ldots ,v_{n-1}$, and denote the eigenvalues of $Q|_W$ by $\mu_1\leq \mu_2\leq...\leq\mu_{n-1}$. By the induction hypothesis and Lemma \ref{lem:compare-eigvals}, we obtain
$$Q(v_i,v_i)\geq \frac{\mu_i}{n-1} \geq \frac{\lambda_i}{n-1} \geq \frac{\lambda_i}{n}$$
for $1\leq i\leq n-1$. Finally, for the last vector, we have
$$Q(v_n,v_n)\geq \frac{Q(v_1,v_1)+...+Q(v_n,v_n)}{n}=\frac{tr(Q)}{n}=\frac{\lambda_1+...+\lambda_n}{n}\geq \frac{\lambda_n}{n}$$
This completes the proof of the lemma.
\end{proof}

\begin{lem}\label{lem:Gram-Schmidt}
Let $Q$ be a positive definite quadratic form on some Euclidean space $V$ of dimension $n$. If $\{v_1,...,v_k\}$ is any $\tau$-orthonormal $k$-frame for
$\tau$ sufficiently small (only depends on n), ordered so that $Q(v_1,v_1)\leq...\leq Q(v_k,v_k)$, then there is an orthonormal $k$-frame $\{u_1,...,u_k\}$
such that $Q(u_i,u_i)\leq 2Q(v_i,v_i)$.
\end{lem}
\begin{proof} We do the Gram-Schmidt process on $\{v_1,...,v_k\}$ and obtain an orthonormal $k$-frame $\{u_1,...,u_k\}$.
Notice $\{v_1,...,v_k\}$ is $\tau$-orthonormal, so we have $u_i=v_i+O(\tau)v_1+...+O(\tau)v_i$, where by
$O(\tau)$ we denote a number that has universal bounded (only depends on $n$) ratio with $\tau$. This implies
$$Q(u_i,u_i)=Q(v_i,v_i)+O(\tau)\sum_{1\leq s\leq t\leq i}Q(v_s,v_t)$$
Since $\abs{Q(v_s,v_t)}\leq \sqrt{Q(v_s,v_s)Q(v_t,v_t)}\leq Q(v_i,v_i)$, we obtain
$$Q(u_i,u_i)\leq Q(v_i,v_i)+O(\tau)Q(v_i,v_i)\leq 2Q(v_i,v_i)$$
for $\tau$ sufficiently small. This completes the proof of the lemma.
\end{proof}

We are now ready to establish Theorem \ref{thm:bound-ratio}.

\begin{proof}
As was shown in \cite[Section 4.4]{CF1}, for any fixed $\epsilon_0 \leq 1/(r+1)$, there are at most $r$ eigenvalues of $Q_2$
that are smaller than $\epsilon_0$ (we will choose $\epsilon_0$ in the course of the proof).
By Lemma \ref{lem:compare-eigvals} the same is true for $Q_2|_S$. We arrange
these small eigenvalues in the order $L_1\leq L_2\leq\ldots \leq L_k$, where $k\leq r$. Observe that, if no such
eigenvalue exists, then by Lemma \ref{lem:compare-eigvals}, $\det(Q_2 |_S)$ is uniformly bounded below, and the
theorem holds (since the eigenvalues of $Q_1|_S$ are all $\leq 1$).
So we will henceforth assume $k\geq 1$.
We denote the corresponding unit eigenvectors by $v_1,...,v_k$ (so that $v_i$ has eigenvalue $L_i$).
Although $V=span\{v_1,...v_k\}$ might not have small angle with $\mathcal{F}$, it is shown in \cite[Section 3]{CF2} that
there is a $k_0\in K$ so that $\angle(k_0v_i,\mathcal{F})\leq 2\epsilon_0^{1/4}$ for each $i$.

Let $\epsilon$ be a constant small enough so that $\epsilon<\delta$, where $\delta$ is from Theorem \ref{thm:weak eigenvalue matching},
and also $\tau:=C'\epsilon$ satisfies the condition of
Lemma \ref{lem:Gram-Schmidt}  (where $C'$ is obtained from Theorem \ref{thm:weak eigenvalue matching}). Hence the choice of $\epsilon$ only
depends on $X$. We now make a choice of $\epsilon_0$ such that $2\epsilon_0^{1/4}<\epsilon$, and hence
$\angle(k_0V,\mathcal{F})< \epsilon$. (Note again the choice of $\epsilon_0$ only depends on $X$.)

Apply Theorem \ref{thm:weak eigenvalue matching} to the frame $\{k_0v_1, \ldots , k_0v_k\}$, and translate the
obtained $(C'\epsilon)$-orthonormal frame by $k_0^{-1}$. This gives us a $(C'\epsilon)$-orthonormal $(2k+r-2)$-frame
$\{v_1',v_1'',...,v_1^{(r)},v_2',v_2'',...,v_k',v_k''\}$, such that for $i=1,...,k$, and $j=1,...r$, we have
$$\angle(hv_i',\mathcal{F}^\perp)\leq C\angle(hv_i,\mathcal{F})$$
$$\angle(hv_i'',\mathcal{F}^\perp)\leq C\angle(hv_i,\mathcal{F})$$
$$\angle(hv_1^{(j)},\mathcal{F}^\perp)\leq C\angle(hv_1,\mathcal{F})$$
for all $h\in K$ (note that we have absorbed the $k_0$-translation into the element $h$).

We notice $\angle(hv_i',\mathcal{F}^\perp)\leq C\angle(hv_i,\mathcal{F})$ implies
$\sin^2(\angle(hv_i',\mathcal{F}^\perp))\leq C_0\sin^2(\angle(hv_i,\mathcal{F}))$ for some $C_0$ depending on $C$. For convenience, we still use $C$ for this new constant. Hence, we obtain
$$Q_1(v_i',v_i')\leq\int_{\partial_F X}\sin^2(\angle (O_\theta^\ast v_i',\mathcal{F}^\perp))d\mu(\theta)$$
$$\leq C\int_{\partial_F X}\sin^2(\angle (O_\theta^\ast v_i,\mathcal{F}))d\mu(\theta)=C Q_2(v_i,v_i)=CL_i$$
An identical estimate gives us $Q_1(v_i'',v_i'')\leq CL_i$, and $Q_1\left(v_1^{(j)},v_1^{(j)}\right)\leq CL_1.$

We rearrange the $(C'\epsilon)$-orthonormal $(2k+r-2)$-frame as
$\{u_1',u_1'',...,u_1^{(r)},u_2',u_2'',...,u_k',u_k''\}$ so that it has increasing order when applying $Q_1$.
Then the inequalities still hold for this new frame:
$$Q_1(u_i',u_i')\leq CL_i$$
$$Q_1(u_i'',u_i'')\leq CL_i$$
$$Q_1\left(u_1^{(j)},u_1^{(j)}\right)\leq CL_1$$
Since the choice of $\epsilon$ makes $C'\epsilon$ satisfy the condition of Lemma \ref{lem:Gram-Schmidt}, we apply
the lemma to this $C'\epsilon$-orthonormal frame. This gives us an orthonormal $(2k+r-2)$-frame
$\left \{\overline{u_1'},\overline{u_1''},...,\overline{u_1^{(r)}},\overline{u_2'},\overline{u_2''},...,\overline{u_k'},\overline{u_k''}
\right \}$,
such that
$$Q_1(\overline{u_i'},\overline{u_i'})\leq 2Q_1(u_i',u_i')\leq 2CL_i$$
$$Q_1(\overline{u_i''},\overline{u_i''})\leq 2Q_1(u_i'',u_i'')\leq 2CL_i$$
$$Q_1\left(\overline{u_1^{(j)}},\overline{u_1^{(j)}})\leq 2Q_1(u_1^{(j)},u_1^{(j)}\right)\leq 2CL_1$$
Again, we can rearrange the orthonormal basis to have increasing order when applying $Q_1$, and it is easy
to check that, for the resulting rearranged orthonormal basis, the same inequalities still hold.

We denote the first $(2k+r-2)$ eigenvalues of $Q_1$ by $\lambda_1'\leq\lambda_1''\leq...\leq\lambda_1^{(r)}\leq\lambda_2'\leq \lambda_2''\leq...\leq \lambda_k'\leq\lambda_k''$, and the first $2k$ eigenvalues of $Q_1\mid _S$ by $\mu_1'\leq\mu_1''\leq...\leq\mu_k'\leq\mu_k''$. Applying Lemma \ref{lem:using-other-bases}, we have
$$\lambda_i'\leq nQ_1(\overline{u_i'},\overline{u_i'})\leq 2nCL_i$$
$$\lambda_i''\leq nQ_1(\overline{u_i''},\overline{u_i''})\leq 2nCL_i$$
$$\lambda_1^{(j)}\leq nQ_1\left(\overline{u_1^{(j)}},\overline{u_1^{(j)}}\right)\leq 2nCL_1$$
for $1\leq i\leq k$ and $1\leq j\leq l$.

Notice $\dim(S)\geq n-r+2$. We apply Lemma \ref{lem:compare-eigvals} and obtain
$$\mu_1'\leq \lambda_1^{(r-1)}\leq 2nCL_1$$
$$\mu_1''\leq \lambda_1^{(r)}\leq 2nCL_1$$
$$\mu_i'\leq \lambda_i'\leq 2nCL_i$$
$$\mu_i''\leq \lambda_i''\leq 2nCL_i$$
for $2\leq i\leq k$.
The eigenvalues of $Q_1|_S$ are bounded above by $1$, and $L_1,...,L_k$ are the only eigenvalues of $Q_2\mid_S$ that are below $\epsilon_0$ (and recall the choice of $\epsilon_0$ only depends on $X$). Therefore,
$$\det(Q_1|_S)\leq \prod_{i=1}^k\mu_i'\mu_i''\leq \prod_{i=1}^k(2nCL_i)^2\leq(2nC)^{2k}\left[\frac{\det (Q_2\mid_S)}{\epsilon_0^{\dim(S)-k}}\right]^2$$
$$\leq \overline{C}\det(Q_2\mid_S)^2$$
where $\overline{C}$ only depends on $X$. This completes the proof of Theorem \ref{thm:bound-ratio}.
\end{proof}

%%%%%%%%%%%%%%%%%%%%%%%%%%%%%%%%%%%%%%%%%%%
%%%%%%%%%%%%%%%%%%%%%%%%%%%%%%%%%%%%%%%%%%%
%%%%%%%%%%%%%%%%%%%%%%%%%%%%%%%%%%%%%%%%%%%

\section{Reduction to the Combinatorial Problem}\label{sec:combinatorial-problem}
In this section, we will prove the ``weak eigenvalue matching'' Theorem \ref{thm:weak eigenvalue matching}, which was
introduced in Section \ref{subsec:eigval-matching}. The approach is to follow \cite{CF2}, and reduce the theorem to a
combinatorial problem. Then we apply Hall's Marriage theorem to solve it.

\subsection{Hall's Marriage Theorem}
We introduce the classic Hall's Marriage Theorem, and later on we will apply a slightly stronger version (Corollary
\ref{cor:Hall's marriage} below) in the proof of Lemma \ref{lem:pick frame}.
\begin{thm}[Hall's Marriage Theorem]\label{thm:Hall's marriage}
Suppose we have a set of $m$ different species $A=\{a_1,..,a_m\}$, and a set of $n$ different planets
$B=\{b_1,...,b_n\}$. Let $\phi: A\rightarrow \mathcal{P}(B)$ be a map which sends a species to the set of all
suitable planets for its survival. Then we can arrange for each species a different planet to survive if and only
if for any subset $A_0\subseteq A$, we have the cardinality inequality $|\phi(A_0)|\geq |A_0|$.

\end{thm}

\begin{cor}\label{cor:two-for-one Hall's marriage}
Under the assumption of Theorem \ref{thm:Hall's marriage}, we can arrange for each species two different planets
if and only if for any subset $A_0\subseteq A$, we have the cardinality inequality $|\phi(A_0)|\geq 2|A_0|$.
\end{cor}
\begin{proof}
Assume there exists such arrangement, the cardinality condition holds obviously. On the other hand, assume we
have the cardinality condition, we want to show there is an arrangement. We make an identical copy on each species
and form the set $A'=\{a_1',...,a_m'\}$. We apply the Hall's Marriage Theorem to the set $A\cup A'$ relative to $B$.
Then for each $i$, both species $a_i$ and $a_i'$ have its own planet, and that means there are two planets for the
original species $a_i$.

To see why the cardinality condition holds, we choose an arbitrary subset $H\cup K'\subseteq A\cup A'$ where
$H \subseteq A$ and $K'\subseteq A'$. Let $K$ be the corresponding identical copy of $K'$ in $A$. We have
$\phi(H\cup K')=\phi(H\cup K)\geq 2|H\cup K|\geq |H|+|K|=|H\cup K'|$. This completes the proof.
\end{proof}

\begin{cor}\label{cor:Hall's marriage}
Suppose we have a set of vectors $V=\{v_1,...,v_r\}$, and for each $v_i$, the selectable set is denoted by
$B_i\subseteq B$. If for any subset $V_0=\{v_{i_1},...,v_{i_k}\}\subseteq V$, we have
$|B_{i_1}\cup...\cup B_{i_k}|\geq 2k+r-2$, then we can pick $(3r-2)$ distinct element
$\left \{b_1',...,b_1^{(r)},b_i',b_i''\;(2\leq i\leq r)\right \}$ in $B$ such that $b_1',...,b_1^{(r)}\in B_1$ and $b_i',b_i''\in B_i$.
\end{cor}

\begin{proof}
First we choose $V_0$ the singleton set that consists of only $v_1$. By hypothesis, we have $|B_1|\geq r\geq (r-2)$,
hence we are able to choose $(r-2)$ elements $b_1^{(3)},...,b_1^{(r)}$ for $v_1$. Next we can easily check the cardinality
condition and apply Corollary \ref{cor:two-for-one Hall's marriage} to the set $V$ with respect to
$B\setminus \left \{b_1^{(3)},...,b_1^{(r)}\right \}$ to obtain the pairs $\{b_i',b_i''\}$ (for each $1\leq i\leq r$).
This completes the proof of this corollary.
\end{proof}

\subsection{Angle inequality}
Throughout this section, we will work exclusively with {\em unit vectors} in $T_xX\simeq \mathfrak{p}$. We embed the point
stabilizer $K_x$ into $\Isom(T_xX)\simeq O(n)$, and endow it with the induced metric. This gives rise to a norm on $K$,
defined by $||k||=\max_{v\in T_xX}\angle{(v,kv)},\forall k\in K$. We denote the Lie algebra of $K_x\simeq K$ by $\mathfrak{k}$,
which has root space decomposition $\mathfrak{k}= \mathfrak{k}_0 \oplus \bigoplus_{\alpha\in \Lambda^+}\mathfrak{k}_\alpha$.
For each small element $k\in K$, the action on a vector $v$ can be approximated by the Lie algebra action, that is, if
$k=\exp(u)$ is small, then $||[u,v]||\approx ||kv-v||\sim \angle{(v,kv)}$, where we write $A\sim B$ if $A/B$ and $B/A$
are both universally bounded. By abuse of notation, we do not distinguish between $||k||$ and $||u||$ inside a very small
neighborhood $\mathcal{U}$ of $0$ inside $\mathfrak{k}$. Although $||\cdot||$ is not linear on $\mathcal{U}$, it is linear up
to a universal constant, that is, $||tu||\sim t||u||$, for all $u\in \mathcal{U}$ and $t$ such that $tu\in \mathcal{U}$.  We
now show the following lemmas.

\begin{lem} \emph{(Compare \cite[Lemma 1.1]{CF2}\label{lem:angle inequality})}
Let $X=G/K$ be a rank $r\geq 2$ irreducible symmetric space of non-compact type, and fix a flat
$\mathcal{F}\subseteq T_xX$ at $x$. Then for any small $\rho>0$, there is a constant $C(\rho)$
with the following property. If $v\in \mathcal{F}$ is arbitrary, and $v^*\in \mathcal{F}$ is a maximally singular vector in
the $\rho$-neighborhood of $v$ (in the sense that the dimension of $K_{v^*}$ is as large as possible), then
$$\angle(hu,\mathcal{F}^\perp)\leq C\angle(hv,\mathcal{F})$$
for any $h\in K$, and $u\in (K_{v^*}\mathcal{F})^\perp\simeq
\bigoplus_{\alpha\in \Lambda^+,\alpha(v^*)\neq 0}\mathfrak{p}_\alpha$,
where $\Lambda^+$ is the set of all positive roots. Moreover, we have
$$\angle(hu,\mathcal{F}^\perp)\leq C\angle(hk_0v,K_{v^*}\mathcal{F})$$
for any $h\in K$, $u\in (K_{v^*}\mathcal{F})^\perp$, and $k_0\in K_{v^*}$.
\end{lem}
\begin{proof}
We only need to verify the inequality when $\angle(hv,\mathcal{F})$ is small. Notice for any vector $v\in \mathcal{F}$,
and any small element $w\in \mathfrak{k}_\alpha= (I+\theta)\mathfrak{g}_\alpha=(I+\theta)(I-\theta)^{-1}\mathfrak{p}_\alpha$,
the Lie algebra action (see Proposition \ref{prop:Lie algebra action}) has norm
\begin{equation}\label{eqn:norm}
||[w,v]||=||-\alpha(v)\cdot(I-\theta)(I+\theta)^{-1}w||\sim |\alpha(v)|\cdot ||w||.
\end{equation}
This is due to the fact that $(I+\theta)(I-\theta)^{-1}$ is a linear isomorphism between $\mathfrak{k}_\alpha$ and
$\mathfrak{p}_\alpha$ (see Proposition \ref{prop:Eberlein}), and when restricted to $\mathfrak{k}_\alpha\cap \mathcal{U}$,
it preserves the norms up to a uniform multiplicative constant.

Infinitesimally speaking, for $h= \exp(w)$, we have that $hv - v = [w, v]$, so the estimate on the Lie algebra
action tells us about the infinitesimal growth of $||hv - v||$. We also see that, since $[w,v]\in \mathfrak{p}_\alpha$, $h$
moves the vector $v$ in the direction $\mathfrak{p}_\alpha$ (which we recall is orthogonal to the flat $\mathcal{F}$,
see Proposition \ref{prop:Eberlein}). Now $v^*$ is a maximally singular vector in the
$\rho$-neighborhood of the unit vector $v$, so once $\rho$ is small enough, if $\alpha$ is any root with
$\alpha(v^*)\neq 0$, then $\alpha (v)$ will be uniformly bounded away from zero (depending only on the choice of $\rho$).
This shows that if a root $\alpha$
satisfies $\alpha(v^*)\neq 0$, then $\angle(hv,\mathcal{F}) \sim ||h||$ for all $h\in \exp(\mathfrak{k}_\alpha\cap \mathcal{U})$.

Now we move to analyzing the general case $h=\exp(w)$, where $w\in \mathfrak k$ is arbitrary.
If $\angle(hv,\mathcal{F})$ is small, then it follows that the components of $hv$ on each
$\mathfrak{p}_\alpha$ must be small. From the discussion above, this implies that the component of
$w$ in each $\mathfrak{k}_\alpha\mid_{\alpha(v^*)\neq 0}$ is small, i.e. $w$ almost lies in
$\mathfrak{k}_{v^*}=\mathfrak{k}_0 \oplus \bigoplus _{\alpha(v^*)=0}\mathfrak{k}_\alpha$. Since
$h$ almost lies in $K_{v^*}$, there exists an element $h_0\in K_{v^*}$ such that $h_0^{-1}h$ is close to the identity.
We write $h=h_0h_1$, where $h_1=\exp(w_1)\in \exp(\mathfrak{k}_{v^*}^\perp) =
\exp\Big(\bigoplus_{\alpha(v^*)\neq 0}\mathfrak{k}_\alpha \Big)$, and observe that the analysis in the previous paragraph
applies to the element $h_1$. Now observe that, infinitesimally,
$h_1v-v=[u_1,v]\in \bigoplus_{\alpha(v^*)\neq 0}\mathfrak{p}_{\alpha}$, so $h_1$ moves $v$ in a direction lying in
$\bigoplus_{\alpha(v^*)\neq 0}\mathfrak{p}_{\alpha}$. On the other hand, infinitesimally, $K_{v^*}$ moves the entire flat
$\mathcal F$ in the directions $\bigoplus_{\alpha(v^*)= 0}\mathfrak{p}_{\alpha}$ (corresponding to the action of its Lie algebra
$\mathfrak{k}_{v_*}$).
But these two directions are orthogonal, which means that $h_1v$ leaves not just $\mathcal{F}$ orthogonally, but actually
leaves orthogonally to the entire orbit $K_{v^*}\mathcal{F}$. This allows us to estimate
\begin{equation}\label{eqn1}
\angle(hv,\mathcal{F})= \angle(h_1v,h_0^{-1}\mathcal{F})\geq \angle(h_1v,K_{v^*}\mathcal{F})\sim ||h_1||,
\end{equation}
where at the last step, we use that $h_1$ moves $v$ orthogonally off the $K_{v^*}$ orbit of $\mathcal F$.
On the other hand, we are assuming that the vector $u$ lies in $(K_{v^*}\mathcal F)^\perp$, hence also in
$h_0^{-1}\mathcal{F}^\perp$. So we have the
sequence of inequalities
\begin{equation}\label{eqn2}
\angle(hu,\mathcal{F}^\perp)=\angle(h_1u,h_0^{-1}\mathcal{F}^\perp)\leq \angle(h_1u,u)\leq ||h_1||.
\end{equation}
Combining equations (\ref{eqn1}) and (\ref{eqn2}) gives us the first inequality.

Similarly, $\angle(hk_0v,K_{v^*}\mathcal{F})$ being small also implies that the component of $h$ on each
$\mathfrak{k}_\alpha\mid_{\alpha(v^*)\neq 0}$ is small. So by writing $h=h_0h_1$ in the same manner, we get
$\angle(hk_0v,K_{v^*}\mathcal{F})=\angle(h_1k_0v,K_{v^*}\mathcal{F})=\angle(k_0^{-1}h_1k_0v,K_{v^*}\mathcal{F})$.
Notice that $K_{v^*}$ conjugates $\mathfrak{k}_{v^*}^\perp$ to itself, so $k_0^{-1}h_1k_0$ is an element in
$\exp(\mathfrak{k}_{v^*}^\perp)$. In view of equation (\ref{eqn:norm}) and the fact that $k_0^{-1}h_1k_0v$ leaves
orthogonally to $K_{v^*}\mathcal{F}$, we obtain
$\angle(k_0^{-1}h_1k_0v,K_{v^*}\mathcal{F})\sim ||k_0^{-1}h_1k_0||=||h_1||$.
Combining this estimate with equation (\ref{eqn2}) gives the second inequality.
\end{proof}

\begin{lem}\label{lem:pick frame}
Let $X=G/K$ be a rank $r\geq 2$ irreducible symmetric space of non-compact type excluding $\SL(3,\mathbb R)/\SO(3)$ and $\SL(4,\mathbb{R})/\SO(4)$,
and fix a flat $\mathcal{F}\subseteq T_xX$ at $x$. Then there exists a constant $C>0$ that only depends on $X$, such that
for any $\frac{1}{2}$-orthonormal r-frame $\{v_1,...,v_r\}$ in $\mathcal{F}$, there is an orthonormal $(3r-2)$-frame
$\left \{v_1',v_1'',...,v_1^{(r)},v_i',v_i'' \hskip 2pt (2\leq i\leq r) \right \}$ in $\mathcal{F}^\perp$ such that
$$\angle(hv_i',\mathcal{F}^\perp)\leq C\angle(hv_i,\mathcal{F})$$
$$\angle(hv_i'',\mathcal{F}^\perp)\leq C\angle(hv_i,\mathcal{F})$$
$$\angle(hv_1^{(j)},\mathcal{F}^\perp)\leq C\angle(hv_1,\mathcal{F})$$
for all $h\in K$, $i=2,...,r$, $j=1,...,r$.
\end{lem}

\begin{proof}
Once we have chosen a parameter $\rho$, we will denote by $v_i^*$ a maximally singular vector in
$\mathcal{F}$ that is $\rho$-close to $v_i$, and we will let
$Q_i=(K_{v_i^*}\mathcal{F})^\perp\simeq \bigoplus_{\alpha\in \Lambda^+,\alpha(v_i^*)\neq 0}\mathfrak{p}_\alpha$.
We now fix an $\rho$ small enough so that, for every
$\frac{1}{2}$-orthonormal $r$-frame $\{v_1, \ldots ,v_r\}\subset \mathcal{F}$, the corresponding
$\{v_i^*\}|_{i=1}^r$ are distinct.
For each $v_i$, the vectors in $Q_i$ are the possible choice of vectors that satisfy the angle inequality provided by
Lemma \ref{lem:angle inequality}. So it suffices to find $r$ vectors in $Q_1$, and two vectors in each $Q_i$ $(i\neq1)$,
such that the chosen $(3r-2)$ vectors form an orthonormal frame.

Now for each root $\alpha$, we pick an orthonormal frame $\{b_{\alpha_i}\}$ on $\mathfrak{p}_\alpha$, we collect them
into the set $B:=\{b_i\}|_{i=1}^{n-r}$, which forms an orthonormal frame on $\mathcal{F}^\perp$.
We will pick the $(3r-2)$-frame from the vectors in $B$. For instance, vector $v_1$ has selectable set
$B_1:=Q_1\cap B$, in which we want to choose $r$ elements, while for $i=2,...,r$, vector $v_i$ has selectable set
$B_i:=Q_i\cap B$, from which we want to choose two elements. Most importantly, the $(3r-2)$ chosen vectors
have to be distinct from each other. This is a purely combinatorial problem, and can be solved by using Hall's
Marriage theorem. In view of Corollary \ref{cor:Hall's marriage}, we only need to check the cardinality condition.
We notice the selectable set of $v_i$ is $B_i$ which spans $Q_i$, so $|B_i|=\dim(Q_i)$. The next lemma will
estimate the dimension of the $Q_i$, and hence will complete the proof of Lemma \ref{lem:pick frame}.
\end{proof}

\begin{lem}\label{lem:dimension-estimate}
Let $X=G/K$ be a rank $r\geq 2$ irreducible symmetric space of non-compact type, excluding $\SL(3,\mathbb R)/\SO(3)$ and $\SL(4,\mathbb{R})/\SO(4)$, and fix a flat $\mathcal{F}$. Assume $\{v_1^*,...,v_r^*\}$ spans $\mathcal{F}$, and let $Q_i=K_{v_i^*}\mathcal{F}$. Then for any subcollection of vectors $\{v_{i_1}^*,...,v_{i_k}^*\}$, we have $\dim(Q_{i_1}+...+Q_{i_k})\geq (2k+r-2)$.
\end{lem}

\begin{proof}
Since $Q_i=(K_{v_i^*}\mathcal{F})^\perp\simeq \bigoplus_{\alpha\in \Lambda^+,\alpha(v_i^*)\neq 0}\mathfrak{p}_\alpha$,
we obtain $Q_{i_1}+...+Q_{i_k}=\bigoplus_{\alpha\in \Lambda^+,\alpha(V)\neq 0}\mathfrak{p}_\alpha$, where
$V=\text{Span}(v_{i_1}^*,...,v_{i_k}^*)$. We can estimate
$$\dim(Q_{i_1}+...+Q_{i_k}) = \sum_{\alpha\in \Lambda^+,\alpha(V)\neq 0}\dim(\mathfrak{p}_\alpha)
\geq \left |\left \{\alpha\in \Lambda^+,\alpha(V)\neq 0\right \}\right |=\frac{1}{2}\left(|\Lambda|-\left |\left \{\alpha\in
\Lambda,H_\alpha\in V^\perp \right \}\right |\right),$$
where $V^\perp$ is the orthogonal complement of $V$ in $\mathcal{F}$, and $H_\alpha$ is the vector in $\mathcal{F}$
that represents $\alpha$.

Now we denote $t_i=\frac{1}{2}\max_{U\subseteq \mathcal{F},\dim(U)=i}\left |\{\alpha\in \Lambda, H_\alpha\in U\}\right |$,
the number of positive roots in the maximally rooted $i$-dimensional subspace. We use the following result that appears in the proof of \cite[Lemma 5.2]{CF1}. For completeness, we also add their proof here.
\begin{claim}\cite[Lemma 5.2]{CF1}
$t_i-t_{i-1}\geq i$, for $1\leq i\leq r-1$.
\end{claim}
\begin{proof} This is proved by induction on $i$. For $i=1$, the inequality holds since $t_0=0$ and $t_1=1$.
Assuming $t_{i-1}-t_{i-2}\geq i-1$ holds, we let $V_{i-1}$ be an $(i-1)$-dimensional maximally rooted subspace.
By definition, the number of roots that lie in $V_{i-1}$ is $2t_{i-1}$. There exists a root $\alpha$ so that $H_\alpha$
does not lie in $V_{i-1}$, and also does not lie on its orthogonal complement (by irreducibility of the root system).
So $H_\alpha^\perp\cap V_{i-1}:=Z$ is a codimension one subspace in $V_{i-1}$. By the induction hypothesis,
there are at least $i-1$ pairs of root vectors that lie in $V_{i-1}-Z$, call them $\pm H_{\alpha_1},...,\pm H_{\alpha_{i-1}}$.
Hence by properties of root system \cite[Proposition 2.9.3]{Eb}, either $\pm (H_\alpha+ H_{\alpha_l})$ or
$\pm (H_\alpha- H_{\alpha_l})$ is a pair of root vectors, for each $1\leq l\leq i-1$. Along with $\pm H_\alpha$,
these pairs of vectors lie in $(V_{i-1}\oplus \langle H_\alpha\rangle)-V_{i-1}$. We have now found $2i$ root vectors
in the $i$-dimensional subspace $V_{i-1}\oplus \langle H_\alpha\rangle$, which do not lie on the maximally rooted
subspace $V_{i-1}$. This shows $t_i-t_{i-1}\geq i$, proving the claim.
\end{proof}

Finally, we can estimate
$\dim(Q_{i_1}+...+Q_{i_k})\geq \frac{1}{2}\left(|\Lambda|-\left | \left \{\alpha\in \Lambda,H_\alpha\in V^\perp\right \}\right |
\right )\geq t_r-t_{r-k}$.
Using the Claim, a telescoping sum gives us $t_r-t_{r-k}\geq r+(r-1)+...+(r-k+1)=k(2r-k+1)/2$, whence the
lower bound
$\dim(Q_{i_1}+...+Q_{i_k})\geq k(2r-k+1)/2$.
When $r\geq 4$, or $k<r=3$ , or $k<r=2$, it is easy to check that $k(2r-k+1)/2 \geq 2k+r-2$.
This leaves the case when $r=k=3$, or $r=k=2$.
When $r=k=3$, we can instead estimate $\dim(Q_1+Q_2+Q_3)=\dim(\mathcal{F}^\perp)=n-3\geq 7=2k+r-2$, provided $n\geq 10$, which only
excludes the rank three symmetric space $\SL(4,\mathbb{R})/\SO(4)$. A similar analysis when $r=k=2$ only excludes the rank two symmetric space
$SL(3,\mathbb R)/SO(3)$. This completes the proof of Lemma \ref{lem:dimension-estimate}, hence completing the proof of Lemma \ref{lem:pick frame}.
\end{proof}

\begin{remark} In the rank two case, both Theorem \ref{thm:weak eigenvalue matching} and
Theorem \ref{thm:bound-ratio} only give you statements about degree $=n$.  Our Main Theorem then only
gives surjectivity of comparison maps in top degree, which agrees with the result of \cite{LS}, and the
corresponding Jacobian estimate is consistent with \cite{CF1} \cite{CF2}.
\end{remark}

\subsection{Proof of Theorem \ref{thm:weak eigenvalue matching}}

We assume $k=r$ without loss of generality since otherwise we can always extend the $k$-frame to an $r$-frame that
has small angle to $\mathcal{F}$. Our first step is to move the frame so as to lie in $\mathcal F$, while
controlling the angles between the resulting vectors (so that we can apply Lemma \ref{lem:pick frame}).
This is done by first moving the vectors to the respective
$K_{v_i^*}\mathcal F$, and then moving to $\mathcal F$.

As in the proof of Lemma \ref{lem:angle inequality}, $\angle(v_i,\mathcal{F})$ being
small implies that the components of $v_i$ on each $\mathfrak{p}_\alpha$ is small. The $K$-orbit of $v_i$
intersects $\mathcal F$ finitely many times (exactly once in each Weyl chamber), and if each of these intersections is
$\rho$-close to a maximally singular vector, choose $v_i^*$ to be the one closest to $v_i$. The element in $K$
moving $v_i$ to $\mathcal F$ will almost lie in $K_{v_i^*}$ (by an argument similar to the one in Lemma \ref{lem:angle inequality}).
By decomposing this element as a product $\hat{k}_ik_i$, we obtain a small $k_i$ which sends $v_i$ to $K_{v_i^*}\mathcal F$
(and $\hat{k}_i \in K_{v_i^*}$).
If $k_i^{-1}=exp(u_i)$, we have
$u_i\in \bigoplus_{\alpha\in \Lambda^+, \alpha(v^*)\neq 0}\mathfrak{k}_\alpha$.

We now estimate the norm $||k_i||$. From the identification of norms in a small neighborhood of the identity, we have
$||k_i|| =||u_i||$. Since $\hat{k}_i$ is an element in $K_{v_i^*}$ that sends $k_iv_i$ to $\mathcal{F}$,
an argument similar to the proof of second inequality in Lemma \ref{lem:angle inequality} gives us
$$\angle(v_i,K_{v_i^*}\mathcal{F})=\angle((\hat{k}_ik_i^{-1}\hat{k}_i^{-1})(\hat{k}_ik_iv_i),K_{v_i^*}\mathcal{F})\sim_{\rho} ||\hat{k}_ik_i^{-1}\hat{k}_i^{-1}||=||k_i||$$
(where the constant will depend on the choice of $\rho$). On the other hand, since
$\mathcal F \subset K_{v_i^*}\mathcal F$, we obtain $\angle(v_i,K_{v_i^*}\mathcal{F})\leq \angle(v_i,\mathcal{F})$.
But by hypothesis, $\angle(v_i,\mathcal{F})<\epsilon$.
Putting all this together, we see that, for each fixed $\rho$, there exists a constant $C'$ that only depends on
$X$, so that each of the $||k_i||$ is bounded above by $\frac{1}{2}C'\epsilon$. In particular, any $\{k_i\}_{i=1}^r$ perturbation
of an orthonormal frame gives rise to a $C'\epsilon$-orthonormal frame, and hence the collection
$\{k_1v_1, \ldots , k_rv_r\}$ forms a $C'\epsilon$-orthonormal frame.

Next, since $\hat{k}_i$ is an element in $K_{v_i^*}$, it leaves $v_i^*$ fixed. From triangle
inequality we obtain
$$ \angle(\hat{k}_ik_iv_i, k_iv_i) \leq  2 \angle(k_iv_i, v_i^*) < 2\rho.$$
It follows that the collection
of vectors $\{\hat{k}_1k_1v_1, \ldots , \hat{k}_rk_rv_r\} \subset \mathcal F$ is obtained from the
$C'\epsilon$-orthonormal frame  $\{k_1v_1, \ldots , k_rv_r\}$ by rotating each of
the various vectors by an angle of at most $2\rho$ hence forms a
$(C'\epsilon + 4\rho)$-orthonormal basis in $\mathcal F$. In particular, once $\rho$ and $\delta$ are
chosen small enough, it gives us a $1/2$-orthonormal basis inside $\mathcal F$.

Applying Lemma \ref{lem:pick frame} to the $1/2$-orthonormal frame $\{\hat{k}_1k_1v_1, \ldots , \hat{k}_rk_rv_r\}
\subset \mathcal F$
gives us an orthonormal $(3r-2)$-frame $\left \{v_1',...,v_1^{(r)},v_i',v_i'' \hskip 2pt (2\leq i\leq r)\right \}$ such that the
angle inequalities hold. Now by the second inequality of Lemma \ref{lem:angle inequality}, we have the following
inequalities:
$$\angle(hv_i',\mathcal{F}^\perp)\leq C\angle(hk_iv_i,K_{v_i^*}\mathcal{F})\leq C\angle(hk_iv_i,\mathcal{F})$$
$$\angle(hv_i'',\mathcal{F}^\perp)\leq C\angle(hk_iv_i,K_{v_i^*}\mathcal{F})\leq C\angle(hk_iv_i,\mathcal{F})$$
$$\angle(hv_1^{(j)},\mathcal{F}^\perp)\leq C\angle(hk_1v_1,K_{v_1^*}\mathcal{F})\leq C\angle(hk_1v_1,\mathcal{F})$$
for $2\leq i\leq r$, $1\leq j\leq r$ and any $h\in K$. Finally we translate each of the vectors $v_i'$, $v_i''$ by $k_i^{-1}$,
and each $v_1^{(j)}$ by $k_1^{-1}$, producing a $C'\epsilon$-orthonormal $(3r-2)$-frame that satisfies the
inequalities in Theorem \ref{thm:weak eigenvalue matching}, hence completing the proof.

%%%%%%%%%%%%%%%%%%%%%%%%%%%%%%%%%%%%%%%%%%%
%%%%%%%%%%%%%%%%%%%%%%%%%%%%%%%%%%%%%%%%%%%
%%%%%%%%%%%%%%%%%%%%%%%%%%%%%%%%%%%%%%%%%%%

\section{Surjectivity of the comparison map in bounded cohomology}\label{sec:surjectivity}

In this Section, we provide some background on cohomology (see Section \ref{sec:cohomology}),
establish the {\bf Main Theorem} (Section \ref{subsec:maintheorem}), establish some limitations
on our technique of proof (Section \ref{subsec:obstruction}), and work
out a detailed class of examples (Section \ref{subsec:example}).

\subsection{Bounded cohomology}\label{sec:cohomology}
Let $X=G/K$ be a symmetric space of non-compact type, and $\Gamma$ be a cocompact lattice in $G$. We recall the definition of group cohomology, working with $\mathbb{R}$ coefficients (so that we can relate these to the de Rham cohomology). Let $C^n(\Gamma,\mathbb{R})=\{f:\Gamma^n\rightarrow \mathbb{R}\}$ be the space of $n$-cochains. Then the coboundary map $d:C^n(\Gamma,\mathbb{R})\rightarrow C^{n+1}(\Gamma,\mathbb{R})$ is defined by
$$df(\gamma_1,...,\gamma_{n+1})=f(\gamma_2,...,\gamma_{n+1})+\sum_{i=1}^n(-1)^if(\gamma_1,...\gamma_{i-1},\gamma_i\gamma_{i+1},\gamma_{i+2},...,\gamma_{n+1})$$
$$+(-1)^{n+1}f(\gamma_1,...,\gamma_n)$$
The homology of this chain complex is $H^\ast(\Gamma,\mathbb{R})$, the group cohomology of $\Gamma$ with
$\mathbb{R}$ coefficients.
Moreover, if we restrict the cochains above to {\it bounded} functions, we obtain the space of bounded $n$-cochains $C^n_b(\Gamma,\mathbb{R})=\{f:\Gamma^n\rightarrow \mathbb{R}\mid f\,\textrm{is bounded}\}$ and the corresponding bounded cohomology $H^\ast_b(\Gamma,\mathbb{R})$ of $\Gamma$. The inclusion of the bounded cochains into the ordinary cochains induces the comparison map $H^\ast_b(\Gamma,\mathbb{R}) \rightarrow H^\ast(\Gamma,\mathbb{R})$.

Similarly, we can define the (bounded) continuous cohomology of $G$, by taking the space of continuous $n$-cochains
$C^n_c(G,\mathbb{R})=\{f:G^n\rightarrow \mathbb{R}\mid f\,\textrm{is continuous}\}$ or the space of bounded continuous
cochains $C^n_{c,b}(G,\mathbb{R})=\{f:G^n\rightarrow \mathbb{R}\mid f\,\textrm{is continuous and bounded}\}$. With the
same coboundary maps as above, this gives two new chain complexes, whose homology will be denoted by $H^\ast_c(G,\mathbb{R})$ and $H^\ast_{c,b}(G,\mathbb{R})$ respectively. Again, one has a naturally induced comparison map
$H^\ast_{c,b}(G,\mathbb{R})\rightarrow H^\ast_{c}(G,\mathbb{R})$.

Now let $M=X/\Gamma$ be the closed locally symmetric space covered by $X$. Note that $M$ is a $K(\Gamma,1)$, so
$$H_{dR}^\ast(M, \mathbb{R})\simeq H_{sing}^\ast(M,\mathbb{R})\simeq H^\ast(\Gamma,\mathbb{R})$$
The isomorphism between the de Rham cohomology and group cohomology is explicitly given by
$$\phi:H_{dR}^k(M, \mathbb{R})\rightarrow H^k(\Gamma,\mathbb{R})$$
$$\omega\mapsto f_\omega$$
where $f_\omega(\gamma_1,\ldots ,\gamma_k)=\int_{\Delta(\gamma_1, \ldots ,\gamma_k)}\widetilde{\omega}$. Here, $\widetilde{\omega}$ is a lift of $\omega$ to $X$, and $\Delta(\gamma_1,\ldots ,\gamma_k)$ is any natural $C^1$ $k$-filling with ordered vertices $\{x,\gamma_1x,(\gamma_1\gamma_2)x,\ldots , \hskip 2pt (\gamma_1\gamma_2\cdots\gamma_k)x\}$ for some fixed basepoint $x\in X$ (for instance, one can choose $\Delta(\gamma_1,\ldots ,\gamma_k)$ to be the geodesic coning simplex, see
Dupont \cite{Du}). Alternatively, we can use the barycentric straightened $C^1$ simplex $st(\Delta(\gamma_1,\ldots ,\gamma_k))$ (which we defined in Section \ref{sec:barycenter method}). That is to say, if we define $\overline{f_\omega}(\gamma_1,\ldots ,\gamma_k)=\int_{st(\Delta(\gamma_1,\ldots ,\gamma_k))}\widetilde{\omega}$,
then $\overline{f_\omega}$ represents the same cohomology class as $f_\omega$. This is due to the fact that the barycentric straightening is $\Gamma$-equivariant (see \cite[Section 3.2]{LS}). We call $\overline{f_\omega}$ the barycentrically straightened cocycle.

On the other hand, there is a theorem of van Est \cite{Ve} which gives the isomorphism between the relative Lie algebra cohomology $H^\ast(\mathfrak{g},\mathfrak{k},\mathbb{R})$ and the continuous bounded cohomology $H^\ast_c(G,\mathbb{R})$. A class in $H^k(\mathfrak{g},\mathfrak{k},\mathbb{R})$ can be expressed by an alternating $k$-form $\varphi$ on $\mathfrak{g}/\mathfrak{k} \simeq T_xX$. By left translation, it gives a closed $C^\infty$ $k$-form $\widetilde{\varphi}$ on $X=G/K$. In \cite{Du}, this isomorphism is explicitly given by
$$\phi:H^k(\mathfrak{g},\mathfrak{k},\mathbb{R})\rightarrow H^k_c(G,\mathbb{R})$$
$$\varphi\mapsto f_\varphi$$
where $f_\varphi(g_1,\ldots ,g_k)=\int_{\Delta(g_1,\ldots ,g_k)}\widetilde{\varphi}$, and $\Delta(g_1,\ldots ,g_k)$ is the geodesic simplex with ordered vertices consisting of $\{x,g_1x,(g_1g_2)x,\ldots ,(g_1g_2\cdots g_k)x\}$ for some fixed basepoint $x\in X$. Again, we can replace $\Delta(g_1,\ldots ,g_k)$ by the barycentric straightened $C^1$ simplex $st(\Delta(g_1,\ldots ,g_k))$, and the resulting barycentrically straightened function $\overline{f_\varphi}(g_1,\ldots ,g_k)=\int_{st(\Delta(g_1,\ldots ,g_k))}\widetilde{\varphi}$ is in the same cohomology class as $f_\varphi$.

\subsection{Proof of the Main Theorem}\label{subsec:maintheorem}

In this section, we use Theorem \ref{thm:bound-ratio} to establish the {\bf Main Theorem}. We need to show both comparison maps $\eta$ and $\eta'$ are surjective. Let us start with $\eta$. We use the van Est isomorphism (see Section \ref{sec:cohomology}) to identify $H^\ast_c(G,\mathbb{R})$ with $H^\ast(\mathfrak{g},\mathfrak{k},\mathbb{R})$. For any class $[f_\varphi]\in H^k_c(G,\mathbb{R})$ where $f_\varphi(g_1,\ldots ,g_k)=\int_{\Delta(g_1,\ldots ,g_k)}\widetilde{\varphi}$, we instead choose the barycentrically straightened representative $\overline{f_\varphi}$. Then for any $(g_1,\ldots ,g_k)\in G^k$, we have
\begin{equation} \label{eqn:Jacobian-bound}
\abs{\overline{f_\varphi}(g_1,\ldots ,g_k)}=\abs{\int_{st(\Delta(g_1,\ldots ,g_k))}\widetilde{\varphi}\hskip 3pt }
\leq \abs{\int_{\Delta_s^k}st_V^\ast\widetilde{\varphi}\hskip 3pt }\leq \int_{\Delta_s^k}\abs{Jac(st_V)}\cdot \|\widetilde{\varphi}\|d\mu_0
\end{equation}
where $d\mu_0$ is the standard volume form of $\Delta_s^k$. But from Proposition \ref{prop:reduction} and Theorem \ref{thm:bound-ratio}, the expression
$|Jac(st_V)|$ is uniformly bounded above by a constant (independent of the choice of vertices $V$ and the point $\delta\in \Delta_s^k$), while the form $\widetilde{\varphi}$ is invariant under the $G$-action, hence bounded in norm. It follows that the last
expression above is less than some constant $C$ that depends only on the choice of alternating form $\varphi$. We have thus
produced, for each class $[f_\varphi]$ in $H^k_c(G,\mathbb{R})$, a bounded representative $\overline{f_\varphi}$.
So the comparison map $\eta$ is surjective. The argument for surjectivity of $\eta'$ is virtually identical, using the
explicit isomorphism between $H^k(\Gamma,\mathbb{R})$ and $H_{dR}^k(M, \mathbb{R})$ discussed in Section
\ref{sec:cohomology}. For any class $[f_\omega]\in H^k(\Gamma,\mathbb{R})$, we choose the barycentrically
straightened representative $\overline{f_\omega}$. The differential form $\widetilde{\omega}$ has bounded norm,
as it is the $\Gamma$-invariant lift of the smooth differential form $\omega$ on the compact manifold $M$. So again,
the estimate in Equation (\ref{eqn:Jacobian-bound}) shows the representative $\overline{f_\omega}$ is bounded,
completing the proof.

\subsection{Obstruction to Straightening Methods}\label{subsec:obstruction}
In this section, we give a general obstruction to the straightening method that is applied in section
\ref{subsec:maintheorem}. In the next section, we will use this to give some concrete examples showing
that Theorem \ref{thm:bound-ratio} {\em is not true}
when $\dim(S)\leq n-r$. Throughout this section, we let $X=G/K$ be an $n$-dimensional symmetric space of
non-compact type, and we give the following definitions.

\begin{defn}\label{def:straightening} Let $C^0(\Delta^k,X)$ be the set of singular k-simplices in $X$, where
$\Delta^k$ is assumed to be equipped with a fixed Riemannian metric. Assume that we are given a collection of
maps $st_k:C^0(\Delta^k,X)\rightarrow C^0(\Delta^k,X)$. We say this collection of maps forms a {\em straightening}
if it satisfies the following properties:
\begin{enumerate}
\item the maps induces a chain map, that is, it commutes with the boundary operators.
\item $st_n$ is $C^1$ smooth, that is, the image of $st_n$ lies in $C^1(\Delta^n,X)$.
\end{enumerate}
For a subgroup $H\leq G$, we say the straightening is $H$-equivariant if the maps $st_k$ all commute with the
$H$-action.
\end{defn}

Since $X$ is simply connected, property $(a)$ of Definition \ref{def:straightening} implies that the chain map
$st_*$ is actually chain homotopic to the identity. Also, property (b) of Definition \ref{def:straightening} implies the image
of any straightened $k$-simplex is $C^1$-smooth, i.e. $Im(st_k) \subset C^1(\Delta^k, X)$. The barycentric straightening
introduced in Section \ref{sec:barycenter method} is a $G$-equivariant straightening. As we saw in Section
\ref{subsec:maintheorem}, obtaining a uniform control on the Jacobian of the straightened $k$-simplices immediately
implies a surjectivity result for the comparison map from bounded cohomology to ordinary cohomology. This motivates
the following:

\begin{defn} We say the straightening is $k$-bounded, if there exists a constant $C>0$, depending only on $X$ and
the chosen Riemannian metric on $\Delta^k$, with the following property. For any $k$-dimensional singular simplex
$f\in C^0(\Delta^k,X)$, and corresponding straightened simplex $st_k(f):\Delta^k\rightarrow X$, the Jacobian of $st_k(f)$
satisfies:
$$\abs{Jac(st_k(f))(\delta)}\leq C$$
where $\delta \in \Delta^k$ is arbitrary (and the Jacobian is computed relative to the fixed Riemannian metric on $\Delta^k$).
\end{defn}

Our Theorem \ref{thm:bound-ratio} and Proposition \ref{prop:reduction} then tells us that, when
$r = \mathbb R\text{-rank}(G)\geq 2$ (excluding the two cases
$\SL(3,\mathbb R)/\SO(3)$ and $\SL(4,\mathbb{R})/\SO(4)$),
our barycentric straightening is $k$-bounded for all $k\geq n-r+2$. One can wonder whether
this range can be improved. In order to obtain obstructions, we recall \cite[Theorem 2.4]{LS}.
Restricting to the case of locally symmetric spaces of non-compact type, the theorem says:

\begin{thm}\cite[Theorem 2.4]{LS}\label{thm:simplicial volume}
Let $M$ be an $n$-dimensional locally symmetric space of non-compact type, with universal cover $X$, and
$\Gamma$ be the fundamental group of $M$. If $X$ admits an $n$-bounded, $\Gamma$-equivariant straightening,
then the simplicial volume of $M$ is positive.
\end{thm}

\begin{cor}\label{cor:no-straightening}
If $X$ splits off an isometric $\mathbb{R}$-factor, then $X$ does not admit an n-bounded, $G$-equivariant straightening.
\end{cor}
\begin{proof}
Let $X\simeq X_0\times \mathbb{R}$ for some symmetric space $X_0$. If $X$ admits an $n$-bounded,
$G$-equivariant straightening, then consider a closed manifold $M \simeq M_0\times S^1$, where
$\widetilde{M_0}\simeq X_0$. According to Theorem \ref{thm:simplicial volume}, the simplicial volume $||M||$ is positive.
But on the other hand $||M||=||M_0\times S^1||\leq C\cdot||M_0||\cdot||S^1||=0$. This contradiction completes the proof.
\end{proof}

We will use subspaces satisfying Corollary \ref{cor:no-straightening} to obstruct bounded straightenings.

\begin{defn}
For $X$ a symmetric space of non-compact type, we define the {\em splitting rank} of $X$, denoted $\text{srk}(X)$,
to be the maximal dimension of a totally geodesic submanifold $Y\subset X$ which splits off an isometric
$\mathbb{R}$-factor.
\end{defn}

For the irreducible symmetric spaces of non-compact type, computations of the splitting rank can be found
in a recent paper by the second author \cite{Wang} (see also Berndt and Olmos \cite{BO} for 
some related work).

\begin{thm}\label{thm:obstruction}
If $k=\text{srk}(X)$, then $X$ does not admit any $k$-bounded, $G$-equivariant straightening.
\end{thm}

\begin{proof}
We show this by contradiction. Assume $X=G/K$ admits a $k$-bounded, $G$-equivariant straightening $st_i$, and
let $Y\subset X$ be a $k$-dimensional totally geodesic subspace which splits isometrically as $Y^\prime \times \mathbb R$.
Denote by $p:X\rightarrow Y$ the orthogonal projection from $X$ to $Y$, and note that the composition $p\circ st_*$
is a straightening on $Y$, which we denote by $\overline{st_*}$. Notice $Y$ is also a symmetric space and can be
identified with $G_0/K_0$, for some $G_0<G$, and $K_0<K$. Then the straightening $\overline{st_*}$ is certainly
$G_0$-equivariant. We claim it is also $k$-bounded. This is because the projection map $p$ is volume-decreasing, hence
$$\abs{Jac(\overline{st_k}(f))}=\abs{Jac\Big(p\big(st_k(f)\big)\Big)}\leq \abs{Jac\big(st_k(f)\big)}\leq C$$
for any $f\in C^0(\Delta^k, X)$. 
Therefore, we conclude that $Y$ admits a $G_0$-equivariant, $k$-bounded straightening. This contradicts 
Corollary \ref{cor:no-straightening}.
\end{proof}

\begin{remark}
In view of Proposition \ref{prop:reduction} and the arguments in Section \ref{subsec:maintheorem}, we
can view Theorem \ref{thm:obstruction} as {\em obstructing} the bounded ratio Theorem \ref{thm:bound-ratio}.
Specifically, if $k=\text{srk}(X)$, then Theorem \ref{thm:obstruction} tells us that one has a sequence
$f_i: \Delta ^k_s \rightarrow X$ with the property that the Jacobian of $st_k(f_i)$ is unbounded. From the definition
of our straightening maps $st_k$, this means one has a sequence $V_i=\left \{v_0^{(i)}, \ldots v_k^{(i)}\right \}\subset X$
of $(k+1)$-tuples of points (the vertices of the singular simplices $f_i$), and a sequence of points
$\delta_i =\left (a_0^{(i)}, \ldots, a_k^{(i)}\right )$ inside the spherical simplex
$\Delta^k_s \subset \mathbb R ^{k+1}$, satisfying
the following property. If one looks at the corresponding sequence of points
$$p_i:= \left(st_k(f_i)\right)(\delta_i) = Bar\left(\sum_{j=0}^k a_j^{(i)}\mu\Big(v_j^{(i)}\Big)\right),$$
one has a sequence of $k$-dimensional subspace $S_i \subset T_{p_i}X$ (given by the tangent spaces
$D(st_{V_i})(T_{\delta_i}\Delta^k_s)$ to the straightened simplex $st_k(f_i)$ at the point $p_i$), and the
sequence of ratios $\det(Q_1|_{S_i})^{1/2}/\det(Q_2|_{S_i})$ tends to infinity. It is not too hard to see that, for each
dimension $k' \leq k$, one can find a $k'$-dimensional subspace $\bar S_i \subset S_i$ such that the sequence
of ratios of determinants, for the quadratic forms restricted to the $\bar S_i$, must also tend to infinity. Thus the
bounded ratio Theorem \ref{thm:bound-ratio} fails whenever $k^\prime \leq \text{srk}(X)$.
\end{remark}

\subsection{The case of $SL(m, \mathbb R)$}\label{subsec:example}
We conclude our paper with a detailed discussion of the special case of the Lie group $G=SL(m, \mathbb R)$, $m\geq 5$.
The continuous cohomology has been computed (see e.g. \cite[pg. 299]{Fu}) and can be described as follows. If $m=2k$
is even, then $H^*_c\left (SL(2k, \mathbb R)\right )$ is an exterior algebra in $k$ generators in degrees
$5, 9, \ldots  , 4k-3, 2k$. If $m=2k+1$ is even, then $H^*_c\left (SL(2k+1, \mathbb R)\right )$ is an exterior algebra in
$k$ generators in degrees $5, 9, \ldots  , 4k+1$.

The associated symmetric space is $X= SL(m, \mathbb R) /SO(m)$, and we have that
$$n=\dim (X) = \dim \left (SL(m, \mathbb R)\right ) - \dim \left (SO(m)\right ) = (m^2-1) - \frac{1}{2}m(m-1)
= {m+1 \choose 2} -1,$$
while the rank of the symmetric space is clearly $r=m-1$. Thus, our {\bf Main Theorem} tells us that, for these Lie
groups, the comparison map
$$H^*_{c,b}(SL(m, \mathbb R)) \rightarrow H^*_{c}(SL(m, \mathbb R))$$
is surjective within the range of degrees $*\geq {m+1\choose 2} - m + 2$.

Observe that the exterior product of all the generators $H^*_c(SL(m,\mathbb R))$ yields the generator for the
top-dimensional cohomology, which lies in degree ${m+1 \choose 2} - 1$. Dropping off the $5$-dimensional
generator in the exterior product yields a non-trivial class in degree ${m+1\choose 2} - 6$. Comparing with the
surjectivity range in our {\bf Main Theorem}, we see that the first interesting example occurs in the case of
$SL(8, \mathbb R)$, where our results imply that $H_{c,b}^{30}\Big(SL(8, \mathbb R)\Big)\neq 0$ (as well as
$H_{c,b}^{35}\Big(SL(8, \mathbb R)\Big)\neq 0$, which was previously known). Of course,
as $m$ increases, our method provides more and more non-trivial bounded cohomology classes. For example,
once we reach $SL(12, \mathbb R)$, we get new non-trivial bounded cohomology classes in
$H_{c,b}^{68}\Big(SL(12, \mathbb R)\Big)$ and $H_{c,b}^{72}\Big(SL(12, \mathbb R)\Big)$.

Finally, let us consider Theorem \ref{thm:obstruction} in the special case of $X=SL(m,\mathbb R)/SO(m)$.
Choose a maximally singular direction in the symmetric space $X$, and let $X_0$
be the set of geodesics that are parallel to that direction. Without loss of generality, we can take $X_0=G_0/K_0$, where
$$G_0=\Bigg\{
         \begin{bmatrix}
           A & 0 \\
           0 & a \\
         \end{bmatrix}
\hskip 5pt | \hskip 5pt \det(A)\cdot a=1,\hskip 2pt a>0\Bigg\}$$
and $K_0=\SO(m)\cap G_0$. Moreover, $X_0$ clearly splits off an isometric $\mathbb{R}$-factor, and can be
isometrically identified with $\SL(m-1,\mathbb{R})/\SO(m-1)\times \mathbb{R}$. This is the maximal dimensional subspace
of $SL(m, \mathbb R)$ that splits off an isometric $\mathbb R$-factor (see \cite[Table 3]{BO}), and the splitting rank is just
$\dim(X_0) = {m \choose 2}$. So in this special case, Theorem
\ref{thm:obstruction} tells us that our method for obtaining bounded cohomology classes {\it fails} once we reach
degrees $\leq {m \choose 2}$. Comparing this to the range where our method
works, we see that, in the special case where $G=SL(m,\mathbb R)$, the only degree which remains unclear
is ${m\choose 2} + 1$. This example shows our {\bf Main Theorem} is very close to the optimal possible.

%%%%%%%%%%%%%%%%%%%%%%%%%%%%%%%%%%%%%%%%%%%
%%%%%%%%%%%%%%%%%%%%%%%%%%%%%%%%%%%%%%%%%%%
%%%%%%%%%%%%%%%%%%%%%%%%%%%%%%%%%%%%%%%%%%%

\section{Concluding remarks}

As we have seen, the technique used in our {\bf Main Theorem} seems close to optimal, at least when restricted
to the Lie groups $SL(m, \mathbb R)$. Nevertheless, the authors believe that for other families of symmetric spaces,
there are likely to be improvements on the range of dimensions in which a barycentric straightening is bounded.

We also note that it might still be possible to bypass the limitations provided by the splitting rank. Indeed, the splitting
rank arguments show that the barycentric straightening is not $k$-bounded, when $k= \text{srk}(X)$. But the barycentric
straightening might still be $k'$-bounded for some $k'<\text{srk}(X)$ (even though the bounded Jacobian Theorem
\ref{thm:bound-ratio} must fail for $k'$-dimensional subspaces).

%\bibliography{references}

\end{document}